\documentclass{amsart}[11pt]
\usepackage{amsmath,amssymb,pifont,calc,bm}
\usepackage{graphicx}
\usepackage{ifpdf}
\usepackage{color}

\include{jdg-p-1.cls}

\newtheorem{theorem}{Theorem}[section]
\newtheorem{proposition}[theorem]{Proposition}
\newtheorem{lemma}[theorem]{Lemma}

\newtheorem*{theorem61}{Theorem 6.1}

\theoremstyle{definition}

\theoremstyle{remark}
\newtheorem{remark}[theorem]{Remark}
\newtheorem*{remarks*}{Remarks}
\newtheorem*{remark*}{Remark}

\numberwithin{equation}{section}

\newcommand{\RR}{{\mathbb R}}

\newcommand{\hb}{{\hbar}}

\newcommand{\trace}{\mbox{\textrm trace}}
\newcommand{\supp}{\mbox{\textrm supp}}

\newcommand{\Vol}{\mbox{\textrm Vol}}

\newcommand{\p}{\mbox{\textrm p}}

\begin{document}
\openup3pt

\title{Semiclassical Spectral Invariants for Schr\"odinger Operators}

\author{Victor Guillemin}\address{Department of Mathematics\\
Massachusetts Institute of Technology \\Cambridge, MA 02139 \\
USA}\thanks{Victor Guillemin is supported in part by NSF grant
DMS-0408993.}\email{vwg@math.mit.edu}
\author{Zuoqin Wang}
\address{Department of Mathematics\\
Johns Hopkins University \\ Baltimore, MD  21218\\
USA} \email{zwang@math.jhu.edu}

\date{\today}

\begin{abstract}
In this article we show how to compute the semi-classical
spectral measure associated with the Schr\"odinger operator on
$\mathbb R^n$, and, by examining the first few terms in the 
asymptotic expansion of this measure, obtain inverse spectral 
results in one and two dimensions. (In particular we show that for 
the Schr\"odinger operator on $\mathbb R^2$ with a radially symmetric
electric potential, $V$, and magnetic potential, $B$, both
$V$ and $B$ are spectrally determined.) We also show that in one 
dimension there is a very simple explicit identity relating the 
spectral measure of the Schr\"odinger operator with its Birkhoff canonical form.
\end{abstract}

\maketitle




\section{Introduction}
\label{sec:1}

Let
  \begin{equation}
  \label{Schr}
  S_\hb = -\frac {\hb^2}2 \Delta + V(x),
  \end{equation}
be the semi-classical Schr\"odinger operator with potential function,
$V(x) \in C^\infty(\RR^n)$, where $\Delta$ is the Laplacian operator
on $\mathbb R^n$. We will assume that $V$ is
nonnegative and that for some $a>0$, $V^{-1}([0,a])$ is compact.
By Friedrich's theorem these assumptions imply that the spectrum of $S_\hb$ on the
interval $[0,a)$ consists of a finite number of discrete eigenvalues
  \begin{equation}
  \label{eig}
  \lambda_i(\hb), \quad 1 \le i \le N(\hb),
  \end{equation}
with $N(\hb) \to \infty$ as $\hb \to 0$. We will show that for $f \in C^\infty(\mathbb R)$,
with $\supp(f) \subset (-\infty, a)$, one has an asymptotic expansion
  \begin{equation}
  \label{tr}
  (2\pi h)^{n} \sum_i f(\lambda_i(\hb)) \sim \sum_{k=0}^\infty \nu_k(f)\hb^{2k}
  \end{equation}
with principal term
  \begin{equation}
  \label{nu0}
  \nu_0(f) = \int f(\frac{\xi^2}2+V(x))\ dx\ d\xi
  \end{equation}
and subprincipal term
  \begin{equation}
  \label{nu1}
  \nu_1(f) = -\frac 1{24} \int f^{(2)}(\frac{\xi^2}2+V(x))\sum_i\frac{\partial^2 V}{\partial x_i^2}\ dx\ d\xi.
  \end{equation}
We will also give an algorithm for computing the higher order terms and will show that the $k^{th}$ term is given
by an expression of the form
  \begin{equation}
  \label{nuk}
  \nu_k(f) = \int\sum_{j=[\frac k2+1]}^k f^{(2j)}(\frac{\xi^2}2+V(x))p_{k,j}
  (DV, \cdots, D^{2k}V) dxd\xi
  \end{equation}
where $p_{k,j}$ are universal polynomials, and $D^kV$ the $k^{th}$ partial derivatives of $V$.
(We will illustrate in an appendix how this algorithm works by computing a few of these terms in the one-dimensional case.)

One way to think about the result above is to view the left hand side of (\ref{tr}) as
defining a measure, $\mu_\hb$, on the interval $[0,a)$, and the right hand side as
an asymptotic expansion of this spectral measure as $\hb \to 0$,
  \begin{equation}
  \label{mea}
  \mu_\hb \sim \sum \hb^{2k} \left(\frac d{dt}\right)^{2k} \mu_k,
  \end{equation}
where $\mu_k$ is a measure on $[0,a)$ whose singular support is the set of critical value
of the function, $V$. This ``semi-classical" spectral theorem is a special case of a
semi-classical spectral theorem for elliptic operators which we will describe in \S 2, and in
\S 3 we will derive the formulas (\ref{nu0}) and (\ref{nu1}) and the algorithm
for computing (\ref{nuk}) from this more general result. More explicitly, we'll show
that this more general result gives, more or less immediately, an expansion similar to (\ref{mea}),
but with a ``$(\frac{d}{dt})^{4k}$" in place of the $(\frac{d}{dt})^{2k}$. We'll then show how to
deduce (\ref{mea}) from this expansion by judicious integrations by parts.

In one dimension our results are closely related to recent results of \cite{Col05}, \cite{Col08}, \cite{CoG}, and \cite{Hez}.
In particular, the main result of \cite{CoG} asserts that if $c \in [0, a)$ is an isolated critical value of $V$
and $V^{-1}(a)$ is a single non-degenerate critical point, $p$, then the first two terms in (\ref{mea})
determine the Taylor series of $V$ at $p$, and hence, if $V$ is analytic in a neighborhood of $p$, determine $V$ itself in
this neighborhood of $p$. In \cite{Col08} Colin de Verdiere proves a number of much stronger variants of
this result (modulo stronger hypotheses on $V$). In particular, he shows that for a single-well potential the spectrum of
$S_\hb$ determines $V$ up to $V(x) \leftrightarrow V(-x)$ without 
any analyticity assumptions provided one makes certain asymmetry assumption on $V$. 
His proof is based on a close examination of the principal
and subprincipal terms in the ``Bohr-Sommerfeld rules to all orders" formula that he derives in \cite{Col05}. However,
we'll show in \S 4 that this result is also easily deducible from the one-dimensional versions of (\ref{nu0}) and (\ref{nu1}),
and as a second application of (\ref{nu0}) and (\ref{nu1}), we will prove in \S 5 an inverse result for symmetric double well
potentials. We will also show (by slightly generalizing a counter-example of Colin) that if one drops his asymmetry assumptions
one can construct uncountable sets, $\{V_\alpha, \alpha \in (0, 1)\}$, of single-well potentials, the $V_\alpha$'s all distinct, 
for which the $\mu_k$'s in (\ref{mea}) are the same, i.e. which are isospectral modulo $O(\hb^\infty)$.

In one dimension one can also interpret the expansion (\ref{mea}) from a somewhat different perspective. In \S 6
we will prove the following ``quantum Birkhoff canonical form" theorem:
\begin{theorem61}
If $V$ is a simple single-well potential on the interval $V^{-1}\left([0,a)\right)$ then on this interval
$S_\hb$ is unitarily equivalent to an operator of the form
  \begin{equation}
  \label{HQB}
  H_{QB}(S_\hb^{har}, \hb^2) + O(\hb^\infty)
  \end{equation}
where $S_\hb^{har}$ is the semi-classical harmonic oscillator: the 1-D Schr\"odinger operator with
potential, $V(x)=\frac {x^2}2$.
\end{theorem61}
Then in \S 7 we will show that the spectral measure, $\mu_\hb$, on the interval, $(0,a)$, is given
by
  \begin{equation}
  \label{nuh}
  \mu_\hb(f) = \int_0^a f(t) \frac{dK}{dt}(t, \hb^2)\ dt
  \end{equation}
where
  \begin{equation}
  \label{sK}
  H_{QB}(s, \hb^2) = t\; \Longleftrightarrow\; s = K(t, \hb^2).
  \end{equation}
In other words, the spectral measure determines the
Birkhoff canonical forms and vice-versa.

The last part of this paper is devoted to studying analogues of results (\ref{tr})-(\ref{mea}) in the presence of magnetic field. In this case the Schrodinger operator becomes
  \begin{equation}
  \label{magS}
  S_\hb^{(m)} = \sum_{k=1}^n (\frac {\hb}i\frac{\partial}{\partial x_k}+a_k(x))^2 + V(x)
  \end{equation}
where $\alpha = \sum a_k dx_k$ is the vector potential associated with the magnetic field and the field itself is the two form
  \begin{equation}
  \label{magf}
  B = d\alpha = \sum B_{ij} dx_i \wedge dx_j.
  \end{equation}
For the operator (\ref{magS}) the analogues of (\ref{tr})-(\ref{mea}) are still true, although the formula (\ref{nuk}) becomes considerately more complicated. We will show that the subprincipal term (\ref{nu1}) is now given by
  \begin{equation}
  \label{sub_m}
  \frac 1{48} \int f^{(2)}(\frac 12 \sum(\xi_i + a_i)^2+V(x)) (-2\sum \frac{\partial^2 V}{\partial x_k^2} +  \|B\|^2)dx d\xi.
  \end{equation}
As a result, we will show in dimension 2 that if $V$ and $B$ are radially symmetric they are spectrally determined.

\subsection*{Acknowledgement}
The results in this paper were, in large part, inspired by conversations
with Shlomo Sternberg. We would like to express to him our warmest thanks.
We would also like to express our warmest thanks to Shijun Zheng for suggestions on Schr\"odinger operators with magnetic fields.


\section{Semi-classical Trace Formula}
\label{sec:2}

Let
  \begin{equation}
  \label{Ph}
  P_{\hb} = \sum_{|\alpha| \le r} a_\alpha(x, \hb) (\hbar D_x)^\alpha
  \end{equation}
be a semi-classical differential operator on $\RR^n$, where
$a_\alpha(x, \hb) \in C^\infty(\RR^n \times \RR)$. Recall that the
Kohn-Nirenberg symbol of $P_\hb$ is
  \begin{equation}
  \label{pKN}
  p(x, \xi, \hb) = \sum_\alpha a_\alpha(x, \hb) \xi^\alpha
  \end{equation}
and its Weyl symbol is
  \begin{equation}
  \label{pW}
  p^w(x, \xi, \hb) = \exp(-\frac{i\hb}2 D_\xi \partial_x) p(x, \xi, \hb).
  \end{equation}
We assume that $p^w$ is a real-valued function, so that $P_\hb$ is
self-adjoint. Moreover we assume that for the interval $[a,b]$, $(p^w)^{-1}([a,b])$,
$0 \le \hb \le h_0$, is compact. Then by Friedrich's theorem, the spectrum of $P_\hb$, $\hb
< h_0$, on the interval $[a,b]$, consists of a finite number of eigenvalues,
  \begin{equation}
  \label{eig}
  \lambda_i(\hb), \quad 1 \le i \le N(\hb),
  \end{equation}
with $N(\hb) \to \infty$ as $\hb \to 0$. Let
  \begin{equation}
  \label{p}
  \p(x, \xi) = p(x, \xi, 0) = p^w(x, \xi, 0),
  \end{equation}
be the principal symbols of $P_\hb$.

Suppose $f \in C^\infty_0(\RR)$ is smooth and compactly supported on $(a,b)$.
Then
  \begin{equation*}
  f(P_\hb) = \frac 1{\sqrt{2\pi}} \int \hat{f}(t) e^{itP_\hb}\ dt,
  \end{equation*}
where $\hat f$ is the Fourier transform of $f$.

\begin{theorem}[\cite{GuS}]
The operator $f(P_\hb)$ is a semi-classical Fourier integral
operator. In the case $p(x, \xi, \hb)=p(x, \xi)$, i.e. $a_\alpha(x, \hb)$ are
independent of $\hb$, $f(P_\hb)$ has the left Kohn-Nirenberg symbol
  \begin{equation}
  \label{bf}
  b_f(x, \xi, \hb) \sim \sum_k \hb^k \left( \sum_{l \le 2k} b_{k,l}(x, \xi)
  \left((\frac 1i \frac d{ds})^l f\right)(\p(x, \xi))\right).
  \end{equation}
\end{theorem}
It follows that
  \begin{equation}
  \label{SemTr}
  \trace f(P_\hb) = \hb^{-n} \int b_f(x, \xi, \hb)\ dx d\xi + O(\hb^\infty).
  \end{equation}

The coefficients $b_{k,l}(x, \xi)$ in (\ref{bf}) can be computed as follows: Let
$Q_\alpha$ be the operator
  \begin{equation}
  \label{Qalpha}
  Q_\alpha = \frac 1{\alpha !}\left(\partial_x + i t \frac{\partial \p}
  {\partial x}\right)^\alpha.
  \end{equation}
Let $b_k(x, \xi, t)$ be defined iteratively by means of the equation
  \begin{equation}
  \label{bm}
  \frac 1i \frac{\partial b_m}{\partial t} = \sum_{|\alpha| \ge 1}
  \sum_{k+|\alpha| = m} (D_\xi^\alpha p)(Q_\alpha b_k),
  \end{equation}
with initial conditions
  \begin{equation}
  \label{b0}
  b_0(x, \xi, t) =1
  \end{equation}
and
  \begin{equation}
  \label{bm0}
  b_m(x, \xi, 0) = 0
  \end{equation}
for $m \ge 1$. Then it is easy to see that $b_k(x, \xi, t)$ is
a polynomial in $t$ of degree $2k$. The functions $b_{k,l}(x,
\xi)$ are just the coefficients of this polynomial,
  \begin{equation}
  \label{bkl}
  b_k(x,\xi,t) = \sum_{l \le 2k} b_{k,l}(x, \xi)t^l.
  \end{equation}
  %


\section{Spectral Invariants for Schr\"odinger Operators}
\label{sec:3}

Now let's compute $\trace f(S_\hb)$ for $f \in
C^\infty_0(-a, a)$ via the semi-classical trace formula
(\ref{SemTr}). Notice that from (\ref{bf}), (\ref{SemTr}) and
(\ref{b0}) it follows that the first trace invariant is
  \begin{equation*}
  \int f(p(x, \xi))\ dxd\xi,
  \end{equation*}
which implies Weyl's law, (\cite{GuS} \S 9.8), for the asymptotic distributions
of the eigenvalues (\ref{eig}).

To compute the next trace invariant, we notice that for the
Schr\"odinger operator (\ref{Schr}),
  \begin{equation}
  \label{pSchr}
  p(x, \xi, \hb) = p_0(x, \xi) = \p(x, \xi) = \frac{\xi^2}2 + V(x),
  \end{equation}
so the operator $Q_\alpha$ becomes
  \begin{equation}
  \label{QalphaSchr}
  Q_\alpha = \frac 1{\alpha !}\left(\partial_x + i t \frac{\partial V}
  {\partial x}\right)^\alpha.
  \end{equation}
It follows from (\ref{bm}) that
  \begin{equation*}
  \aligned
  \frac 1i \frac{\partial b_m}{\partial t} &= \sum_{|\alpha| \ge 1}
  \sum_{k+|\alpha|=m} D^\alpha_\xi \p Q^\alpha b_k \\
  & = \sum_k \frac{\xi_k}i \left(\frac{\partial}{\partial {x_k}}+
  it\frac{\partial V}{\partial x_k}\right)b_{m-1} - \frac 12 \sum_k
  \left(\frac{\partial}{\partial x_k} +i t \frac{\partial V}{\partial x_k}
  \right)^2 b_{m-2}.
  \endaligned
  \end{equation*}
Since $b_0(x, \xi, t)=1$ and $b_1(x, \xi, 0)=0$, we have
  \begin{equation*}
  b_1(x, \xi, t) = \frac{i t^2}2\sum_l \xi_l \frac{\partial V}{\partial x_l},
  \end{equation*}
and thus
  \begin{equation*}
  \aligned
  \frac 1i \frac{\partial b_2}{\partial t} & = \sum_k \frac{\xi_k}i
  \left(\frac{\partial}{\partial_{x_k}}+it\frac{\partial V}{\partial x_k}
  \right)(\frac{i t^2}2\sum_l \xi_l \frac{\partial V}{\partial x_l}) -
   \frac 12 \sum_k \left(\frac{\partial}{\partial x_k} +i t \frac{\partial V}
   {\partial x_k}\right)^2 (1) \\
  & = \frac{t^2}2 \sum_{k,l}\xi_k \xi_l\left(\frac{\partial^2 V}
  {\partial x_k \partial x_l} + i t \frac{\partial V}{\partial x_k}
  \frac{\partial V}{\partial x_l}\right) - \frac 12 \sum_k \left( it
  \frac{\partial^2 V}{\partial x_k^2} - t^2 \frac{\partial V}{\partial x_k}
  \frac{\partial V}{\partial x_k}\right).
  \endaligned
  \end{equation*}
It follows that
  \begin{equation}
  \label{b2}
  b_2(x, \xi, t) = \frac{t^2}4 \sum_k \frac{\partial^2 V}{\partial
  x_k^2} + \frac{i t^3}6 \left( \sum_k (\frac{\partial V}{\partial
  x_k})^2 + \sum_{k,l} \xi_k \xi_l \frac{\partial^2 V}{\partial x_k
  \partial x_l}\right) -\frac{t^4}8 \sum_{k,l} \xi_k \xi_l \frac
  {\partial V}{\partial x_k} \frac{\partial V}{\partial x_l}.
  \end{equation}
Thus the next trace invariant will be the integral
  \begin{equation}
  \label{2ndTerm}
  \aligned
  \int -& \frac 14 \sum_k \frac{\partial^2 V}{\partial
  x_k^2} f''(\frac{\xi^2}2+V(x)) - \frac 16 \sum_k (\frac{\partial V}{\partial
  x_k})^2 f^{(3)}(\frac{\xi^2}2+V(x)) \\ & - \frac 16 \sum_{k,l} \xi_k \xi_l \frac{\partial^2 V}{\partial x_k
  \partial x_l}  f^{(3)}(\frac{\xi^2}2+V(x))  -\frac 18 \sum_{k,l} \xi_k \xi_l \frac
  {\partial V}{\partial x_k} \frac{\partial V}{\partial x_l} f^{(4)}(\frac{\xi^2}2+V(x))\ dxd\xi.
  \endaligned
  \end{equation}
We can apply to these expressions the integration by parts formula,
  \begin{equation}
  \label{IBP1}
  \int \frac{\partial A}{\partial x_k} B(\frac{\xi^2}2+V(x))\ dxd\xi =
  - \int A(x)\frac{\partial V}{\partial x_k} B'(\frac{\xi^2}2 +
  V(x))\ dxd\xi
  \end{equation}
and
  \begin{equation}
  \label{IBP2}
  \int \xi_k\xi_l A(x)B'(\frac{\xi^2}2+V(x))\ dxd\xi = - \int
  \delta_k^l A(x) B(\frac{\xi^2}2 +V(x))\ dxd\xi.
  \end{equation}

Applying (\ref{IBP1}) to the first term in (\ref{2ndTerm}) we get
  \begin{equation*}
  \int \frac 14 \sum_k (\frac{\partial V}{\partial x_k})^2
  f^{(3)}(\frac{\xi^2}2+V(x))\ dxd\xi,
  \end{equation*}
and by applying (\ref{IBP2}) the fourth term in (\ref{2ndTerm})
becomes
  \begin{equation*}
  \int \frac 18 \sum_k (\frac{\partial V}{\partial x_k})^2
  f^{(3)}(\frac{\xi^2}2+V(x))\ dxd\xi.
  \end{equation*}
Finally applying both (\ref{IBP2}) and (\ref{IBP1}) the third term
in (\ref{2ndTerm}) becomes
  \begin{equation*}
  \int -\frac 16 \sum_k (\frac{\partial V}{\partial x_k})^2
  f^{(3)}(\frac{\xi^2}2+V(x))\ dxd\xi.
  \end{equation*}
So the integral (\ref{2ndTerm}) can be simplified to
  \begin{equation*}
  \frac 1{24} \int \sum_k (\frac{\partial V}{\partial
  x_k})^2 f^{(3)}(\frac{\xi^2}2+V(x))\ dxd\xi.
  \end{equation*}
We conclude
\begin{theorem}
\label{main} The first two terms of (\ref{SemTr}) are
  \begin{equation}
  \label{TwoTerm} \trace f(S_\hb) = \int f(\frac{\xi^2}2+V(x))\
  dxd\xi + \frac 1{24}\hb^2 \int \sum_k (\frac{\partial V}{\partial
  x_k})^2 f^{(3)}(\frac{\xi^2}2+V(x)) \ dxd\xi+ O(\hb^4).
  \end{equation}
\end{theorem}
In deriving (\ref{TwoTerm}) we have assumed that $f$ is compactly supported. However,
since the spectrum of $S_\hb$ is bounded from below by zero the left and right hand
sides of (\ref{TwoTerm}) are unchanged if we replace the ``$f$" in (\ref{TwoTerm}) by
\emph{any} function, $f$, with support on $(-\infty, a)$, and, as a consequence of this remark,
it is easy to see that the following two integrals,
  \begin{equation}
  \label{1st}
  \int_{\frac{\xi^2}2+V(x) \le \lambda} \ dxd\xi
  \end{equation}
and
  \begin{equation}
  \label{2nd}
  \int_{\frac{\xi^2}2+V(x) \le \lambda} \sum_k (\frac{\partial V}{\partial x_k})^2 dxd\xi
  \end{equation}
are spectrally determined by the spectrum (\ref{eig}) on the interval $[0, a]$. Moreover, from
(\ref{TwoTerm}), one reads off the Weyl law: For $0<\lambda<a$,
  \begin{equation}
  \label{Weyl}
  \#\{\lambda_i(\hb) \le \lambda \}  = (2 \pi \hb)^{-n}\left( \Vol(\frac{\xi^2}2+V(x)\le \lambda) + O(\hb)\right).
  \end{equation}

We also note that the second term in the formula (\ref{TwoTerm}) can, by (\ref{IBP2}), be written in the form
  \begin{equation*}
  \frac 1{24} \hb^2 \int \sum_{k} \frac{\partial^2 V}{\partial x_k^2} f^{(2)}(\frac{\xi^2}2+V(x))\ dxd\xi
  \end{equation*}
and from this one can deduce an $\hb^2$-order ``cumulative shift to the left" correction to the Weyl law.

\subsection{Proof of (\ref{nuk})}

To prove (\ref{nuk}), we notice that for $m$ even, the lowest
degree term in the polynomial $b_m$ is of degree $\frac m2 + 1$, thus we can write
  \begin{equation*}
  b_m = \sum_{l=-\frac m2+1 }^{m} b_{m,l}t^{m+l}.
  \end{equation*}
Putting this into the the iteration formula, we will get
  \begin{equation*}
  \aligned
  \frac {m+l}{i}b_{m,l} = & \sum \frac{\xi_k}{i}\frac{\partial b_{m-1,l}}{\partial x_k} + \sum \xi_k \frac{\partial V}{\partial x_k} b_{m-1, l-1} -\frac 12 \sum \frac{\partial^2 b_{m-2,l+1}}{\partial x_k^2} \\
  & -\frac i2 (\frac{\partial}{\partial x_k}\frac{\partial V}{\partial x_k} +\frac{\partial V}{\partial x_k}\frac{\partial}{\partial x_k})b_{m-2,l} +\frac 12 \sum (\frac{\partial V}{\partial x_k})^2 b_{m-2,l-1},
  \endaligned
  \end{equation*}
from which one can easily conclude that for $l \ge 0$,
  \begin{equation}
  \label{bml}
  b_{m,l} = \sum \xi^\alpha (\frac{\partial V}{\partial x})^\beta p_{\alpha, \beta}(DV, \cdots, D^mV)
  \end{equation}
where $p_{\alpha, \beta}$ is a polynomial, and $|\alpha|+|\beta| \ge 2l-1$. It follows that, by applying the integration by parts formula (\ref{IBP1}) and (\ref{IBP2}), all the $f^{(m+l)}$, $l \ge 0$, in the integrand of the $\hb^n$th term in the expansion (\ref{bf}) can be replaced by $f^{(m)}$. In other words, only derivatives of $f$ of degree $\le 2k$ figure in the expression for $\nu_k(f)$. For those terms involving derivatives of order less than $2k$, one can also use integration by parts to show that each $f^{(m)}$ can be replaced by a $f^{(m+1)}$ and a $f^{(m-1)}$. In particular, we can replace all the odd derivatives by even derivatives. This proves (\ref{nuk}).


\section{Inverse Spectral Result: Recovering the Potential Well}
\label{sec:4}

Suppose $V$ is a ``potential well", i.e. has a unique
nondegenerate critical point at $x=0$ with minimal value $V(0)=0$,
and that $V$ is increasing for $x$ positive, and decreasing for $x$
negative. For simplicity assume in addition that
  \begin{equation}
  \label{der}
  -V'(-x) > V'(x)
  \end{equation}
holds for all $x$. We will show how to use the spectral invariants
(\ref{1st}) and (\ref{2nd}) to recover the potential function
$V(x)$ on the interval $|x|<a$.

\begin{figure}[h]
\centering
  \setlength{\unitlength}{0.0335 mm}%
  \begin{picture}(3755.6, 1434.7)(0,0)
  \put(0,0){\includegraphics[width=1\columnwidth]{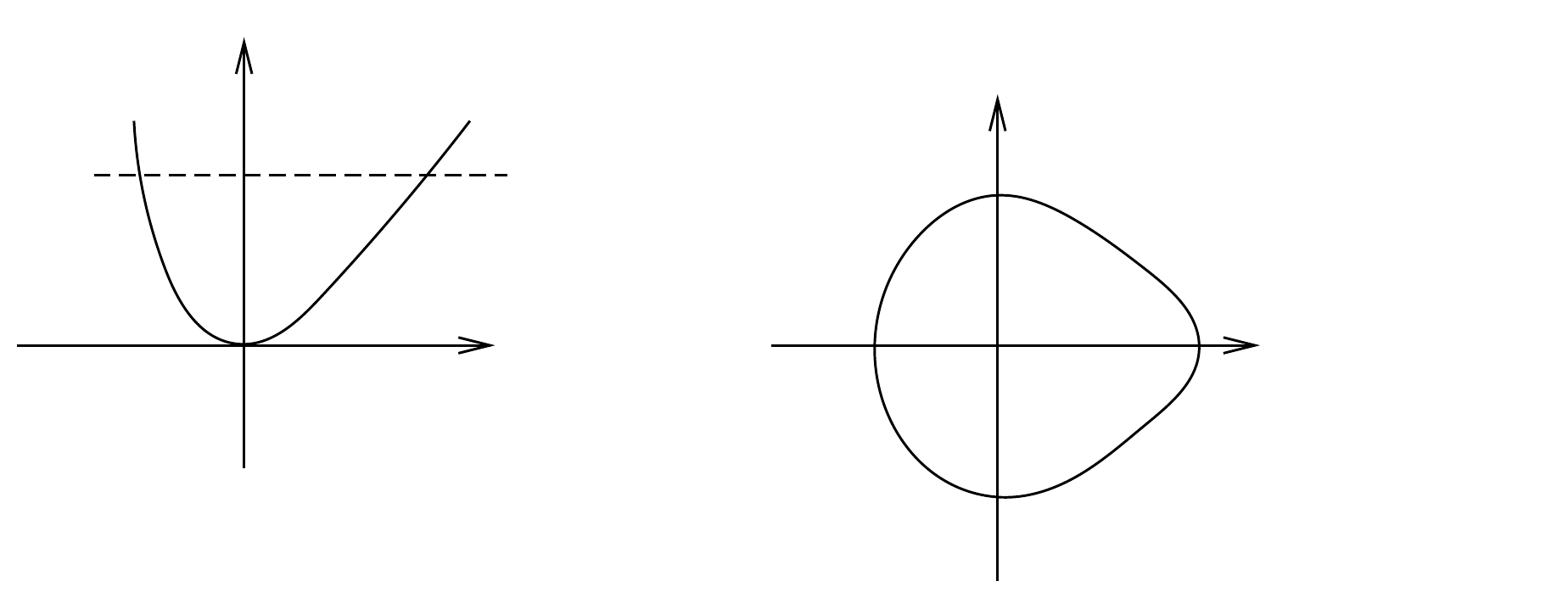}}
  \put(608.46,1302.02){\fontsize{12.80}{15.36}\selectfont \makebox(135.0, 90.0)[l]{$y$\strut}}
  \put(1171.84,500.86){\fontsize{12.80}{15.36}\selectfont \makebox(135.0, 90.0)[l]{$x$\strut}}
  \put(2967.91,487.00){\fontsize{12.80}{15.36}\selectfont \makebox(135.0, 90.0)[l]{$x$\strut}}
  \put(2413.15,1199.29){\fontsize{12.80}{15.36}\selectfont \makebox(225.0, 90.0)[l]{$\xi$\strut}}
  \put(2707.65,832.51){\fontsize{10.24}{12.29}\selectfont \makebox(1008.0, 72.0)[l]{$\frac{\xi^2}2+V(x)=\lambda$\strut}}
  \put(1128.98,1164.68){\fontsize{10.24}{12.29}\selectfont \makebox(288.0, 72.0)[l]{$y=V(x)$\strut}}
  \put(618.73,1020.85){\fontsize{10.24}{12.29}\selectfont \makebox(324.0, 72.0)[l]{$\lambda$\strut}}
  \put(1854.75,525.25){\fontsize{10.24}{12.29}\selectfont \makebox(540.0, 72.0)[l]{$-x_2(\lambda)$\strut}}
  \put(2660.31,521.55){\fontsize{10.24}{12.29}\selectfont \makebox(504.0, 72.0)[l]{$x_1(\lambda)$\strut}}
  \put(2207.68,692.10){\fontsize{10.24}{12.29}\selectfont \makebox(180.0, 72.0)[l]{$A_2$\strut}}
  \put(2498.76,678.41){\fontsize{10.24}{12.29}\selectfont \makebox(180.0, 72.0)[l]{$A_1$\strut}}
  \end{picture}%
\caption{\label{SWP}Single Well Potential}
\end{figure}
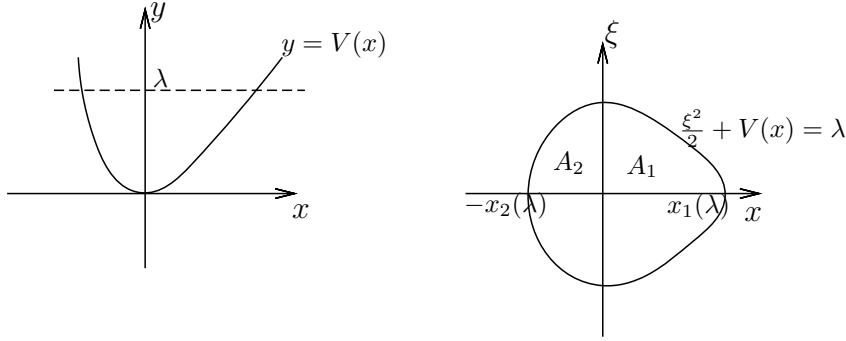

For $0<\lambda<a$ we let $-x_2(\lambda)<0<x_1(\lambda)$ be the
intersection of the curve $\frac{\xi^2}2+V(x) = \lambda$ with the
$x$-axis on the $x-\xi$ plane. We will denote by $A_1$ the region in
the first quadrant bounded by this curve, and by $A_2$ the region in
the second quadrant bounded by this curve. Then from (\ref{1st})
and (\ref{2nd}) we can determine
  \begin{equation}
  \label{1st2}
  \int_{A_1}+ \int_{A_2}\  dxd\xi
  \end{equation}
and
  \begin{equation}
  \label{2nd2}
  \int_{A_1}+ \int_{A_2}\ V'(x)^2 dxd\xi.
  \end{equation}
Let $x=f_1(s)$ be the inverse function of $s=V(x), x \in (0,a)$.
Then
  \begin{equation*}
  \aligned
  \int_{A_1} V'(x)^2\ dxd\xi
  & = \int_0^{x_1(\lambda)} V'(x)^2
      \int_0^{\sqrt{2(\lambda-V(x))}}d\xi dx \\
  & = \int_0^{x_1(\lambda)} V'(x)^2 \sqrt{2\lambda-2V(x)}\ dx \\
  & = \int_0^\lambda \sqrt{2\lambda - 2s }V'(f_1(s))\ ds \\
  & = \int_0^\lambda \sqrt{2\lambda - 2s} \left(\frac{df_1}{ds}\right)^{-1} ds.
  \endaligned
  \end{equation*}
Similarly
  \begin{equation*}
  \int_{A_2} V'(x)^2\ dxd\xi = \int_0^\lambda \sqrt{2\lambda - 2s}
  \left(\frac{df_2}{ds}\right)^{-1} \ ds,
  \end{equation*}
where $x=f_2(s)$ is the inverse function of $s=V(-x), x \in (0,
a)$. So the spectrum of $S_\hb$ determines
  \begin{equation}
  \label{2nd3}
  \int_0^\lambda \sqrt{\lambda  - s}
  \left((\frac{df_1}{ds})^{-1}+(\frac{df_2}{ds})^{-1}\right)\ ds.
  \end{equation}
Similarly the knowledge of the integral (\ref{1st2}) amounts to
the knowledge of
  \begin{equation}
  \label{1st3}
  \int_0^\lambda \sqrt{\lambda  - s}
  \left(\frac{df_1}{ds}+\frac{df_2}{ds}\right)\ ds.
  \end{equation}
Recall now that the fractional integration operation of Abel,
  \begin{equation}
  \label{FracInt}
  J^a g(\lambda) = \frac 1{\Gamma(a)}\int_0^\lambda (\lambda -t)^{a-1} g(t)\ dt
  \end{equation}
for $a>0$ satisfies $J^a J^b = J^{a+b}$. Hence if we apply $J^{1/2}$ to the
expression (\ref{1st3}) and (\ref{2nd3}) and then differentiate by $\lambda$ two times we recover
$\frac{df_1}{ds}+\frac{df_2}{ds}$  and
$(\frac{df_1}{ds})^{-1}+(\frac{df_2}{ds})^{-1}$ from the spectral
data. In other words, we can determine $f_1'$ and $f_2'$ up to the ambiguity $f_1'
\leftrightarrow f_2'$.

However, by (\ref{der}), $f_1'>f_2'$. So we can from the above determine $f_1'$ and $f_2'$, and hence $f_i, i=1,2.$ So we conclude
\begin{theorem}
\label{InvSpe} Suppose the potential function $V$ is a potential
well, then the semi-classical spectrum of $S_\hb$ modulo $o(\hb^2)$
determines $V$ near $0$ up to $V(x) \leftrightarrow V(-x)$.
\end{theorem}

\begin{remark}
The hypothesis (\ref{der}) or some ``asymmetry" condition similar to it is necessary for the theorem above to be true. 
To see this we note that since $V(x)$ and $V(-x)$ have the same spectrum the integrals in (\ref{nuk}) have to be 
invariant under the involution, $x \to -x$. (This is also easy to see directly from the algorithm (\ref{bm}).) Now let
$V: \mathbb R \to \mathbb R$ be a single well potential satisfying $V(0)=0$, $V(x) \to +\infty$ as $x \to \pm \infty$ and
\[
(a) \qquad \qquad V(-x)=V(x) \mbox{\ for\ }k \ge 0 \mbox{\ and\ }2k \le x \le 2k+1
\]
and
\[
(b) \qquad \qquad V(-x)<V(x) \mbox{\ for\ }k \ge 0 \mbox{\ and\ }2k+1 < x < 2k+2.
\]
Now write the integral (\ref{nuk}) as a sum
\begin{equation}
\label{sum}
\sum_k \int_{I_k} + \int_{-I_k},
\end{equation}
where $I_k$ is the set, $\{(x, \xi), k \le x \le k+1\}$, and for $\alpha \in (0,1)$ having
the binary expansion $.a_1a_2a_3\cdots$, $a_i=0$ or $1$, let $V_\alpha$ be the potential
\[
V_\alpha(x)=V(x) \mbox{\ on\ } 2k<x<2k+1 \mbox{\ if\ } a_k=0
\]
and
\[
V_\alpha(x)=V(-x) \mbox{\ on\ } 2k<x<2k+1 \mbox{\ if\ } a_k=1.
\]
In view of the remark above the summations (\ref{sum}) are unchanged if we replace $V$ by $V_\alpha$.
\end{remark}

\begin{remark}
The formula (\ref{1st3}) can be used to construct lots of Zoll potentials, i.e. potentials for which the Hamiltonian flow $v_H$ associated with  $H=\xi^2 + V(x)$ is periodic of period $2\pi$. It's clear that the potential $V(x)=x^2$ has this property and is the only even potential with this property. However, by (\ref{1st3}) and the area-period relation (See Proposition 6.1) every single-well potential $V$ for which 
\[f_1(s)+f_2(s)=2 s^{1/2}\] 
has this property. We will discuss some implications of this in a sequel to this paper.
\end{remark}


\section{Inverse Spectral Result: Recovering Symmetric Double Well Potential}

We can also use the spectral invariants above to recover double-well potentials. Explicitly, suppose $V$ is a symmetric double-well potential, $V(x)=V(-x)$, as shown in the below graph. Then $V$ is defined by two functions $V_1$, $V_2$:

\begin{figure}[h]
\centering
  \setlength{\unitlength}{0.05 mm}%
  \begin{picture}(2249.9, 710.3)(0,0)
  \put(0,0){\includegraphics{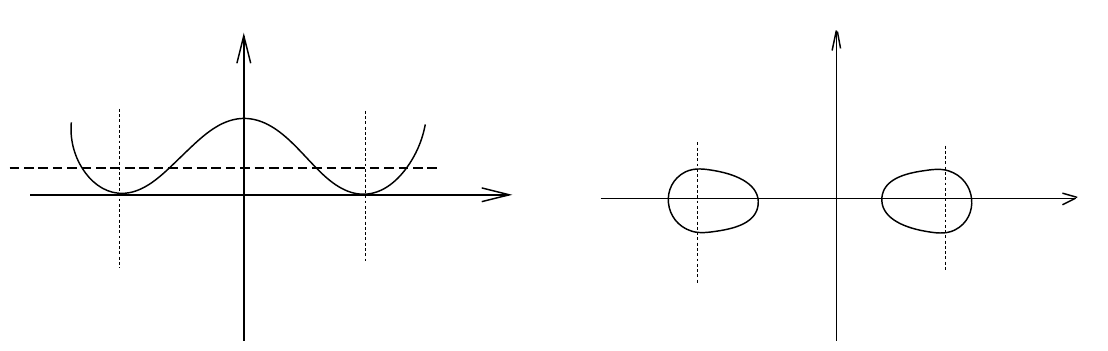}}
  \put(251.01,265.15){\fontsize{5.56}{6.67}\selectfont \makebox(78.1, 39.1)[l]{$-a$\strut}}
  \put(742.21,267.78){\fontsize{5.56}{6.67}\selectfont \makebox(58.6, 39.1)[l]{$a$\strut}}
  \put(1423.81,268.86){\fontsize{5.56}{6.67}\selectfont \makebox(78.1, 39.1)[l]{$-a$\strut}}
  \put(1870.43,265.30){\fontsize{5.56}{6.67}\selectfont \makebox(58.6, 39.1)[l]{$a$\strut}}
  \put(987.76,252.62){\fontsize{5.56}{6.67}\selectfont \makebox(58.6, 39.1)[l]{$x$\strut}}
  \put(2171.81,259.79){\fontsize{5.56}{6.67}\selectfont \makebox(58.6, 39.1)[l]{$x$\strut}}
  \put(1710.20,650.54){\fontsize{5.56}{6.67}\selectfont \makebox(97.7, 39.1)[l]{$\xi$\strut}}
  \put(286.50,393.45){\fontsize{5.56}{6.67}\selectfont \makebox(195.3, 39.1)[l]{$y=V_1(x)$\strut}}
  \put(862.95,405.88){\fontsize{5.56}{6.67}\selectfont \makebox(195.3, 39.1)[l]{$y=V_2(x)$\strut}}
  \put(499.04,376.73){\fontsize{5.56}{6.67}\selectfont \makebox(175.8, 39.1)[l]{$\lambda$\strut}}
  \put(1500.36,378.24){\fontsize{5.56}{6.67}\selectfont \makebox(546.9, 39.1)[l]{$\frac{\xi^2}2+V(x)=\lambda$\strut}}
  \put(515.67,629.48){\fontsize{5.56}{6.67}\selectfont \makebox(58.6, 39.1)[l]{$y$\strut}}
  \end{picture}%
\caption{Double Well Potential}
\end{figure}
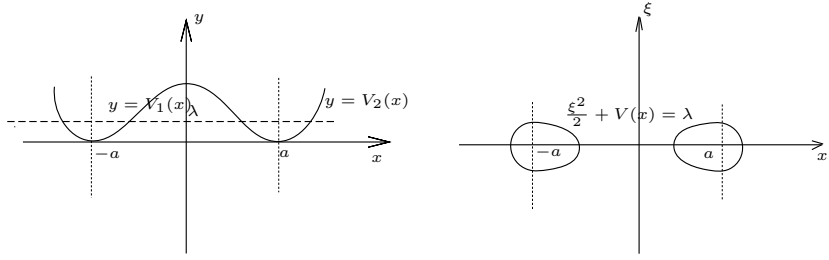

As in the single well potential case, let $f_1$, $f_2$ be the inverse functions
  \begin{equation*}
  x=f_1(s) \; \Longleftrightarrow \; s=V_1(x+a)
  \end{equation*}
and
  \begin{equation*}
  x=f_2(s) \; \Longleftrightarrow \; s=V_2(x-a).
  \end{equation*}
For $\lambda$ small, the region $\{(x,\xi)\ |\ \frac{\xi^2}2+V(x) \le \lambda \}$ is as indicated in figure 2, so if we
apply the same argument as in the previous section, we recover from the area of this region the
sum $\frac{df_1}{ds}+\frac{df_2}{ds}$ via Abel's integral. Similarly from the spectral invariant
$\int_{\frac{\xi^2}2+V(x) \le \lambda} (V')^2\ dxd\xi$ we recover the sum $(\frac{df_1}{ds})^{-1} + (\frac{df_2}{ds})^{-1}$.
As a result, we can determine $V_1$ and $V_2$ modulo an asymmetry condition such as (\ref{der}).

The same idea also shows that if $V$ is decreasing on $(-\infty, -a)$ and is increasing on $(b, \infty)$, and that $V$ is known on $(-a,b)$, then we can recover $V$ everywhere. In particular, we can weaken the symmetry condition on double well potentials: if $V$ is a double well potential, and is symmetric on the interval $V^{-1}[0, V(0)]$, then we can recover $V$.


\section{The Birkhoff Canonical Form Theorem for the 1-D Schr\"odinger Operator}

Suppose that $V^{-1}\left([0,a]\right)$ is a closed interval, $[c,d]$, with $c<0<d$ and
$V(0)=0$. Moreover suppose that on this interval, $V''>0$. We will show below that there exists
a semi-classical Fourier integral operator,
  \begin{equation*}
  \mathcal U: C_0^\infty(\RR) \to C^\infty(\RR)
  \end{equation*}
with the properties
  \begin{equation}
  \label{uf}
  \mathcal U f(S_\hb) \mathcal U^t = f(H_{QB}(S_\hb^{har}, \hb)) + O(\hb^\infty)
  \end{equation}
for all $f \in C_0^\infty\left((-\infty, a)\right)$ and
  \begin{equation}
  \label{uA}
  \mathcal U \mathcal U^t A = A
  \end{equation}
for all semi-classical pseudodifferential operators with microsupport on $H^{-1}\left([0,a)\right)$.
To prove these assertions we will need some standard facts about Hamiltonian systems in two dimensions:
With $H(x, \xi) = \frac{\xi^2}2+V(x)$ as above, let $v=v_H$ be the Hamiltonian vector field
  \begin{equation*}
  v_H = \frac{\partial H}{\partial \xi} \frac{\partial}{\partial x} - \frac{\partial H}{\partial x} \frac{\partial}{\partial \xi}
  \end{equation*}
and for $\lambda < a$ let $\gamma(t, \lambda)$ be the integral curve of $v$ with initial point,
$\gamma(0, \lambda)$, lying on the $x$-axis and $H(\gamma(0, \lambda)) = \lambda$. Then, since $L_vH=0$,
$H(\gamma(t, \lambda))=\lambda$ for all $t$. Let $T(\lambda)$ be the time required for this curve to return to its initial
point, i.e.
  \begin{equation*}
  \gamma(t, \lambda) \ne \gamma(0, \lambda), \; \mbox{for\ } 0<t<T(\lambda)
  \end{equation*}
and
  \begin{equation*}
  \gamma(T(\lambda), \lambda) = \gamma(0, \lambda).
  \end{equation*}
\begin{proposition} [The area-period relation]
\label{apr}
Let $A(\lambda)$ be the area of the set $\{x, \xi\ |\ H(x, \xi)<\lambda\}$. Then
  \begin{equation}
  \label{ap}
  \frac d{d\lambda} A(\lambda) = T(\lambda).
  \end{equation}
\end{proposition}
\begin{proof}
Let $w$ be the gradient vector field
  \begin{equation*}
  \left( (\frac{\partial H}{\partial x})^2 + (\frac{\partial H}{\partial \xi})^2 \right)^{-1}
  \left(\frac{\partial H}{\partial x}\frac{\partial}{\partial x} + \frac{\partial H}{\partial \xi} \frac{\partial}{\partial \xi}\right)
  \rho(H)
  \end{equation*}
where $\rho(t)=0$ for $t < \frac{\varepsilon}2$ and $\rho(t)=1$ for $t>\varepsilon$. Then for
$\lambda>\varepsilon$ and $t$ positive, $\exp(tw)$ maps the set $H=\lambda$ onto the set $H=\lambda+t$ and
hence
  \begin{equation*}
  A(\lambda+t) = \int_{H=\lambda+t}\ dx\ d\xi = \int_{H=\lambda} (\exp{tw})^*\ dx\ d\xi.
  \end{equation*}
So for $t=0$,
  \begin{equation*}
  \frac d{dt}A(\lambda+t) = \int_{H \le \lambda} L_w\ dx\ d\xi = \int_{H \le \lambda} d\iota(w) dxd\xi = \int_{H=\lambda} \iota(w) dxd\xi.
  \end{equation*}
But on $H=\lambda$,
  \begin{equation*}
  \iota(w)dxd\xi =  \left( (\frac{\partial H}{\partial x})^2 + (\frac{\partial H}{\partial \xi})^2 \right)^{-1}
  \left(\frac{\partial H}{\partial x}d\xi - \frac{\partial H}{\partial \xi}dx\right).
  \end{equation*}
Hence by the Hamilton-Jacobi equations
  \begin{equation*}
  dx = \frac{\partial H}{\partial \xi}\ dt
  \end{equation*}
and
  \begin{equation*}
  d\xi = -\frac{\partial H}{\partial x}\ dt,
  \end{equation*}
the right hand side is just $-dt$. So
  \begin{equation*}
  \frac{dA}{d\lambda}(\lambda) = -\int_{H=\lambda}\ dt = T(\lambda).
  \end{equation*}
\end{proof}

For $\lambda=a$, let $c=\frac{A(\lambda)}{2\pi}$ and let
  \begin{equation*}
  H^0_{HB}: [0,c] \to [0, a]
  \end{equation*}
be the function defined by the identities
  \begin{equation*}
  H^0_{HB}(s) = \lambda \; \Longleftrightarrow s=\frac{A(\lambda)}{2\pi}
  \end{equation*}
and let
  \begin{equation*}
  H_{CB}(x, \xi) := H^0_{QB}(\frac{x^2+\xi^2}2).
  \end{equation*}
Thus by definition
  \begin{equation}
  \label{acb}
  A_{CB}(\lambda) = \mbox{area} \{H_{CB}<\lambda\} = A(\lambda).
  \end{equation}
Now let $v$ be the Hamiltonian vector field associated with the Hamiltonian, $H$,
as above and $v_{CB}$ the corresponding vector field for $H_{CB}$. Also as above
let $\gamma(t, \lambda)$ be the integral curve of $v$ on the level set, $H=\lambda$,
with initial point on the $x$-axis, and let $\gamma_{CB}(t, \lambda)$ be the analogous
integral curve of $v_{CB}$. We will define a map of the set $H<a$ onto the set $H_{CB}<a$
by requiring
  \begin{equation}
  \label{fhcb}
  \aligned
  & \mathrm{i.\ } f^* H_{CB} = H, \\
  & \mathrm{ii.\ } f \mbox{\ maps the\ }x\mbox{-axis into itself},\qquad \qquad \qquad\indent\indent\indent\indent\indent\indent \\
  & \mathrm{iii.\ } f(\gamma(t, \lambda)) = \gamma_{CB}(t, \lambda).
  \endaligned
  \end{equation}
Notice that this mapping is well defined by proposition \ref{apr}. Namely by the identity (\ref{acb})
and the area-period relation, the time it takes for the trajectory $\gamma(t, \lambda)$ to circumnavigate
this level set $H=\lambda$ coincides with the time it takes for $\gamma_{CB}(t, \lambda)$ to circumnavigate the level
set $H_{CB}=\lambda$. It's also clear that the mapping defined by (\ref{fhcb}), $\mathrm{i}-\mathrm{iii}$, is a smooth mapping
except perhaps at the origin and in fact since it satisfies $f^*H_{BC}=H$ and $f_*v_H = v_{H_{CB}}$, is a
symplectomorphism. We claim that it is a $C^\infty$ symplectomorphism at the origin as well. This slightly non-trivial
fact follows from the classical Birkhoff canonical form theorem for the \emph{Taylor} series of $f$ at the
origin. (The proof of which is basically just a formal power series version of the proof above. See \cite{GPU}, \S 3
for details.)

Now let $\mathcal U_0: C_0^\infty(\RR) \to C^\infty(\RR)$ be a semi-classical Fourier integral operator quantizing
$f$ with the property (\ref{uA}). By Egorov's theorem $\mathcal U_0 S_\hb \mathcal U_0^t$ is a zeroth order
semi-classical pseudodifferential operator with leading symbol $H^0_{QB}(\frac{x^2+\xi^2}2)$ on the
set $\{(x, \xi)\ |\ H^0_{QB}<a\}$, and hence the operator
  \begin{equation*}
  \mathcal U_0 S_\hb \mathcal U_0^t - H^0(S_\hb^{har})
  \end{equation*}
is semi-classical pseudodifferential operator on this set with leading symbol of order $\hb^2$. We'll show
next that this $O(\hb^2)$ can be improved to an $O(\hb^4)$. To do so, however, we'll need the following lemma:
\begin{lemma}
\label{lemma}
Let $g$ be a $C^\infty$ function on the set $H^{-1}(0,a)$. Then there exists a $C^\infty$
function, $h$, on this set and a function $\rho \in C^\infty(0,a)$ such that
  \begin{equation}
  \label{g}
  g=L_v h + \rho(H).
  \end{equation}
\end{lemma}
\begin{proof}
Let
  \begin{equation*}
  \rho(\lambda)  = \int_0^{T(\lambda)} g(\gamma(t, \lambda))\ dt
  \end{equation*}
and let $g_1 = g - \rho(H)$. Then
  \begin{equation*}
  \int_0^{T(\lambda)} g_1(\gamma(t, \lambda))\ dt = 0.
  \end{equation*}
So one obtains a function $h$ satisfying (\ref{g}) by setting
  \begin{equation*}
  h(\gamma(t, \lambda)) = \int_0^t g_1(\gamma(t, \lambda))\ dt.
  \end{equation*}
\end{proof}
\begin{remark*}
The identity (\ref{g}) can be rewritten as
  \begin{equation}
  \label{g2}
  g=\{H, h\} + \rho(H).
  \end{equation}
\end{remark*}

Now let $-\hb^2 g$ be the leading symbol of
  \begin{equation*}
  S_\hb - \mathcal U_0^t H^0_{QB}(S_\hb^{har}) \mathcal U_0 =: \hb^2 R_0.
  \end{equation*}
Then if $h$ and $\rho$ are the functions (\ref{g}) and $Q$ is a self adjoint
pseudodifferential operator with leading symbol $h$ one has, by (\ref{g2}),
  \begin{equation*}
  \aligned
  \exp(i\hb^2Q)S_\hb\exp(-i\hb^2Q) & = S_\hb + i [Q, S_\hb] \hb^2 + O(\hb^4) \\
  & = S_\hb - \hb^2 (R_0+\rho(S_\hb)) + O(\hb^4).
  \endaligned
  \end{equation*}
Hence if we replace $\mathcal U_0$ by $\mathcal U_1=\mathcal U_0 \exp(i\hb^2 Q)$ we have
  \begin{equation}
  \label{u1s}
  \aligned
  \mathcal U_1 S_\hb \mathcal U_1^t & = H^0_{QB}(S_\hb^{har}) + \hb^2 \rho\left(
  H^0_{QB}(S_\hb^{har})\right) + O(\hb^4) \\
  & = H^0_{QB}(S_\hb^{har}) + H^1_{QB}(S_\hb^{har}) + O(\hb^4)
  \endaligned
  \end{equation}
microlocally on the set $H^{-1}(0,a)$.

As above there's an issue of whether (\ref{u1s}) holds microlocally at the origin as well, or
alternatively: whether, for the $g$ above, the solutions $h$ and $\rho$ of (\ref{g2}) extend
smoothly over $x=\xi=0$. This, however, follows as above from known facts about Birkhoff canonical forms in a formal
neighborhood of a critical point of the Hamiltonian, $H$.

To summarize what we've proved above: ``Quantum Birkhoff modulo $\hb^2$" implies ``Quantum Birkhoff modulo $\hb^4$".
The inductive step, ``Quantum Birkhoff modulo $\hb^{2k}$" implies ``Quantum Birkhoff modulo $\hb^{2k+2}$" is proved in
exactly the same way. We will omit the details.


\section{Birkhoff Canonical Forms and Spectral Measures}

Let $g(s)$ be a $C_0^\infty$ function on the interval $(0, \infty)$. Then by the Euler-Maclaurin
formula
  \begin{equation*}
  \sum_{n=0}^\infty g\left(\hb (n+\frac 12)\right) = \int_0^\infty g(s)\ ds + O(\hb^\infty).
  \end{equation*}
Hence for $f \in C_0^\infty(0, a)$,
  \begin{equation*}
  \trace f(H_{QB}(S_\hb^{har}, \hb)) = \int_0^\infty f(H_{QB}(s, \hb))\ ds + O(\hb^\infty).
  \end{equation*}
Thus by (\ref{uf}) and (\ref{uA}),
  \begin{equation}
  \label{nuhf}
  \nu_\hb(f) = \trace f(S_\hb) = \int_0^\infty f(H_{QB}(s, \hb))\ ds + O(\hb^\infty).
  \end{equation}
Thus if $K(t, \hb)$ is the inverse of the function $H_{QB}(s, \hb)$ on the interval $0<t<a$, i.e.
for $0<t<a$,
  \begin{equation*}
  K(t, \hb)=s \; \Longleftrightarrow H_{QB}(s, \hb)=t,
  \end{equation*}
then (\ref{nuhf}) can be rewritten as
  \begin{equation}
  \label{nuhf2}
  \nu_\hb(f) = \int_0^a f(t)\frac{dK}{dt}\ dt + O(\hb^\infty),
  \end{equation}
or more succinctly as
  \begin{equation}
  \label{nuh2}
  \nu_\hb = \frac{dK}{dt}\ dt.
  \end{equation}
Hence in view of the results of \S 6 this gives one an easy way to recover $H_{QB}(s, \hb)$ from $V$
and its derivatives via fractional integration.


\section{Semiclassical Spectral Invariants for Schr\"odinger Operators with Magnetic Fields}

In this section we will show how the results in \S 3 can be extended to Schr\"odinger operators with magnetic fields.
Recall that a semi-classical Schr\"odinger operator with magnetic field on $\mathbb R^n$ has the form
  \begin{equation}
  \label{sch_m}
  S^m_\hb := \frac 12 \sum_j (\frac {\hb}i\frac{\partial}{\partial x_j} + a_j(x))^2 + V(x)
  \end{equation}
where $a_k \in C^\infty(\mathbb R^n)$ are smooth functions defining a magnetic field $B$, which, in dimension 3 is given by $\vec B = \vec \nabla \times \vec a$, and in arbitrary dimension by the 2-form $B=d(\sum a_kdx_k)$. We will assume that the vector potential $\vec a$ satisfies the Coulomb gauge condition,
  \begin{equation}
  \label{cou}
  \nabla \cdot \vec{a}  = \sum_j \frac{\partial a_j}{\partial x_j} = 0.
  \end{equation}
(In view of the definition of $B$, one can always choose such a Coulomb vector potential.) In this case, the Kohn-Nirenberg symbol of the operator (\ref{sch_m}) is given by
  \begin{equation}
  \label{syl_sch_m}
  p(x, \xi, \hb) = \frac 12\sum_j (\xi_j+a_j(x))^2+V(x).
  \end{equation}
Recall that
  \begin{equation}
  \label{QalphaSchr_m}
  Q_\alpha = \frac 1{\alpha !}\prod_k \left(\frac{\partial}{\partial {x_k}} + i t\frac{\partial p}{\partial x_k}\right)^{\alpha_k},
  \end{equation}
so the iteration formula (\ref{bm}) becomes
  \begin{equation}
  \label{ite_s_m}
  \frac 1i \frac{\partial b_m}{\partial t} =  \sum_k \frac 1i \frac{\partial p}{\partial \xi_k}(\frac{\partial}{\partial x_k}
  + it\frac{\partial p}{\partial x_k})b_{m-1}
   - \frac 12 \sum_k \left(\frac{\partial}{\partial x_k} +i t \frac{\partial p}{\partial x_k}
  \right)^2 b_{m-2}.
  \end{equation}
from which it is easy to see that
  \begin{equation}
  \label{b1_s_m}
  b_1(x, \xi, t) = \sum_k  \frac{\partial p}{\partial \xi_k}\frac{\partial p}{\partial x_k} \frac{it^2}{2}.
  \end{equation}
Thus the ``first" spectral invariant is
  \begin{equation*}
  \int \sum_k (\xi_k +a_k(x))\frac{\partial p}{\partial x_k} f^{(2)}(p)\  dxd\xi
   = - \int \sum_k \frac{\partial a_k}{\partial x_k} f'(p) dx d\xi = 0,
  \end{equation*}
where we used the fact $\sum \frac{\partial a_k}{\partial x_k} = 0$.

With a little more effort we get for the next term
  \begin{equation*}
  \aligned
  b_2(x, \xi, t)
  = & \frac{t^2}4 \sum_k \frac{\partial^2 p}{\partial x_k^2}  \\
  & + \frac{i t^3}6 \left(\sum_{k,l }\frac{\partial p}{\partial \xi_k}\frac{\partial a_l}{\partial x_k}\frac{\partial p}{\partial x_l}  +  \sum_{k,l} \frac{\partial p}{\partial \xi_k} \frac{\partial p}{\partial \xi_l} \frac{\partial^2 p}{\partial x_k \partial x_l} + \sum_k (\frac{\partial p}{\partial x_k})^2 \right) \\
  & + \frac{-t^4}{8}  \sum_{k,l} \frac{\partial p}{\partial \xi_k}\frac{\partial p}{\partial x_k}\frac{\partial p}{\partial \xi_l}\frac{\partial p}{\partial x_l}.
  \endaligned
  \end{equation*}
and, by integration by parts, the spectral invariant
  \begin{equation}
  \label{sp_s_m}
  I_\lambda= -\frac 1{24} \int \left(\sum_k \frac{\partial^2 p}{\partial x_k^2} - \sum_{k,l} \frac{\partial a_k}{\partial x_l} \frac{\partial a_l}{\partial x_k}
  \right) f^{(2)}(p(x, \xi)) dxd\xi.
  \end{equation}
Notice that
  \begin{equation*}
  \frac{\partial^2 p}{\partial x_k^2} = \sum_j \frac{\partial^2 a_j}{\partial x_k^2}\frac{\partial p}{\partial \xi_j} + \sum_j (\frac{\partial a_j}{\partial x_k})^2 + \frac{\partial^2 V}{\partial x_k^2}
  \end{equation*}
and
  \begin{equation*}
  \|B\|^2 = \mathrm{tr}B^2 = 2\sum_{j,k} \frac{\partial a_k}{\partial x_j} \frac{\partial a_j}{\partial x_k} - 2\sum_{j,k}(\frac{\partial a_k}{\partial x_j})^2
  \end{equation*}
So the subprincipal term is given by
  \begin{equation*}
  \frac{1}{48} \int f^{(2)}(p(x, \xi)) \left(\|B\|^2 - 2 \sum_k \frac{\partial^2 V}{\partial x_k^2}\right)\ dx\ d\xi.
  \end{equation*}
Finally Since the spectral invariants have to be gauge invariant by definition, and since any magnetic field has by gauge change a coulomb vector potential representation, the integral
\[
  \int_{p<\lambda} \left(\|B\|^2 - 2 \sum_k \frac{\partial^2 V}{\partial x_k^2}\right)\ dx\ d\xi
\]
is spectrally determined for an arbitrary vector potential. Thus we proved
\begin{theorem}
For the semiclassical Schr\"odinger operator (\ref{sch_m}) with magnetic field $B$, the spectral measure
$
\nu(f) = \mbox{trace}f(S_{\hb}^m)
$
for $f \in C^\infty_0(\mathbb R)$ has an asymptotic expansion
\[
\nu^m(f) \sim (2\pi\hb)^{-n} \sum \nu_r^m(f) \hb^{2r},
\]
where
\[
  \label{nu0m}
\nu_0^m(f) = \int f(p(x, \xi, \hb)) dxd\xi
\]
and
\[
\label{nu1m}
\nu_1^m(f) = \frac 1{48} \int f^{(2)}(p(x, \xi, \hb)) (\|B\|^2 - 2 \sum \frac{\partial^2 V}{\partial x_i^2}).
\]
\end{theorem}


\section{A Inverse Result for The Schr\"odinger Operator with A Magnetic Field}

Making the change of coordinates $(x, \xi) \to (x, \xi+a(x))$, the expressions (\ref{nu0m}) and (\ref{nu1m}) simplify to
\[
\nu_0^m(f) = \int f(\xi^2+V) dxd\xi
\]
and
\[
\label{nu1m}
\nu_1^m(f) = \frac 1{48} \int f^{(2)}(\xi^2+V) (\|B\|^2 - 2 \sum \frac{\partial^2 V}{\partial x_i^2}) dxd\xi.
\]
In other words, for all $\lambda$, the integrals
\[
\mathrm{I}_\lambda = \int_{\xi^2+V(x) < \lambda} dxd\xi
\]
and
\[
\mathrm{I\!I}_\lambda = \int_{\xi^2+V(x) < \lambda} (\|B\|^2 -  2 \sum \frac{\partial^2 V}{\partial x_i^2}) dxd\xi
\]
are spectrally determined.

Now assume that the dimension is 2, so that the magnetic field $B$ is actually a scalar $B = B dx_1 \wedge dx_2$. Moreover, assume that $V$ is a radially symmetric potential well, and the magnetic field $B$ is also radially symmetric. Introducing polar coordinates
\[
\aligned
& x_1^2+x_2^2 = s, \; dx_1 \wedge dx_2 = \frac 12 ds \wedge d\theta \\
& \xi_1^2+\xi_2^2 = t, \; d\xi_1 \wedge d\xi_2 = \frac 12 dt \wedge d\psi
\endaligned
\]
we can rewrite the integral $I_\lambda$ as
\[
\mathrm{I}_\lambda = \pi^2 \int_0^{s{\scriptsize (\lambda)}} (\lambda - V(s))ds,
\]
where $V(s(\lambda))=\lambda$. Making the coordinate change $V(s)=x \Leftrightarrow s=f(x)$ as before, we get
\[
\mathrm{I}_\lambda = \pi^2 \int_0^\lambda (\lambda-x) \frac{df}{dx}dx.
\]
A similar argument shows
\[
\mathrm{I\!I}_\lambda = \pi^2 \int_0^\lambda (\lambda-x) H(f(x)) \frac{df}{dx} dx,
\]
where
\[
H(s) = B(s)^2 - 4s V''(s) - 2 V'(s).
\]
It follows that from the spectral data, we can determine
\[
f'(\lambda) = \frac{1}{\pi^2}\frac{d^2}{d\lambda^2}\mathrm{I}_\lambda
\]
and
\[
H(f(\lambda)) f'(\lambda) = \frac{1}{\pi^2} \frac{d^2}{d\lambda^2}\mathrm{I\!I}_\lambda.
\]
So if we normalize $V(0)=0$ as before, we can recover $V$ from the first equation and $B$ from the second equation.

\begin{remark}
In higher dimensions, one can show by a similar (but slightly more complicated)
argument that $V$ and $\|B\|$ are both spectrally determined if they are radially symmetric.
\end{remark}

\appendix
\section*{Appendix A: More Spectral Invariants in 1-dimension}
\setcounter{equation}{0}
\renewcommand{\thesection}{A}

For simplicity we will only consider the dimension one case. One
can solve the equation (\ref{bm}) for the Schr\"odinger operator
with initial conditions (\ref{b0}) and (\ref{bm0}) inductively,
and get in general
  \begin{equation}
  \label{b2m}
  b_{2m}(x, \xi, t) = \sum_{k=m+1}^{4m} t^k \sum_{n+t=k-m, n \le m \atop l_1+\cdots+l_t=2m}
  \xi^{2n} V^{(l_1)} \cdots V^{(l_t)} a_{n,l},
  \end{equation}
and
  \begin{equation}
  \label{b2m-1}
  b_{2m-1}(x, \xi, t) = \sum_{k=m+1}^{4m-2} t^k \sum_{n+t=k-m, n \le m-1 \atop l_1+\cdots+l_t=2m-1}
  \xi^{2n+1} V^{(l_1)} \cdots V^{(l_t)} \tilde a_{n,l},
  \end{equation}
where $a_{n, l}$ and $\tilde a_{n,l}$ are constants depending on
$n$ and $l_1, \cdots, l_t$. In particular,
  \begin{equation}
  \label{b3}
  \aligned
  b_3(x, \xi, t) = & \frac{t^3}6 \xi V^{(3)}(x) + \frac {t^4}3 i\xi
  \left(V'(x) V''(x)+\frac 18 \xi^2 V^{(3)}\right) \\ & -
  \frac{t^5}{12} \xi \left(
  V'(x)^3 + \xi^2 V'(x)V''(x) \right) - \frac{t^6}{48}i\xi^3V'(x)^3,
  \endaligned
  \end{equation}
and
\begin{equation}
  \label{b4}
  \aligned
  b_4(x, \xi, t) = &  -\frac{t^3}{24} i V^{(4)}(x) + t^4 \left(
  \frac{7}{96} V''(x)^2 + \frac 5{48}V'(x)V^{(3)}(x) + \frac 1{16}
  \xi^2 V^{(4)}(x)\right) \\
  &+ t^5 \left( \frac{13}{120} i V'(x)^2V''(x) + \frac{13}{120}
  i \xi^2 V''(x)^2 + \frac{19}{120}i \xi^2 V'(x)V^{(3)}(x) + \frac 1
  {120}i \xi^4 V^{(4)}(x)\right) \\
  & + t^6 \left( -\frac 1{72} V'(x)^4 - \frac{47}{288} \xi^2 V'(x)^2
  V''(x) - \frac 1{72} \xi^4 V''(x)^2 - \frac 1{48} \xi^4 V'(x)V^{(3)}(x)
  \right) \\
  & - \frac{t^7}{48} \left( i\xi^2 V'(x)^4 + i \xi^4 V'(x)^2V''(x)
  \right) + \frac {t^8}{384} \xi^4 V'(x)^4.
  \endaligned
  \end{equation}
The order $\hb^k$ term is given by integrating the above formula
with $t^k$ replaced by $\frac 1{i^k} f^{(k)}(\frac{\xi^2}2+V(x))$.
By integration by parts
  \begin{equation*}
  \int \xi^{2k}A(x)B^{(k)}(\frac{\xi^2}2+V(x))\ dxd\xi = (-1)^k
  (2k-1)!! \int A(x)B(\frac{\xi^2}2 +V(x))\ dxd\xi,
  \end{equation*}
so we can simplify the integral to
  \begin{equation*}
  \int \left(\frac {V^{(4)}f^{(3)}}{240} + \frac {(V'')^2f^{(4)}}{160}  +
  \frac {V'V'''f^{(4)}}{120}  + \frac {11 (V')^2V''f^{(5)}}{1440} +
  \frac {(V')^4f^{(6)}}{1152}\right) dxd\xi,
  \end{equation*}
Notice that
  \begin{equation*}
  \int V^{(4)}f^{(3)}  = - \int   V'V'''f^{(4)}
   = \int V''V''f^{(4)} + V''V'V'f^{(5)}
  \end{equation*}
and
  \begin{equation*}
  \int V'V'V''f^{(5)} = - \int \left(2 V'V'V''f^{(5)} +
  V'V'V'V'f^{(6)}\right),
  \end{equation*}
we can finally simplify the integral to
  \begin{equation*}
  \int \left(\frac 1{480}(V''(x))^2f^{(4)}(\frac{\xi^2}2+V(x)) +
  \frac 7{3456}(V'(x))^4f^{(6)}(\frac{\xi^2}2+V(x)) \right)\
  dxd\xi,
  \end{equation*}
or
  \begin{equation*}
  \frac 1{288}\int \left(\frac {\xi^4}{5}(V''(x))^2 +
  \frac 7{12}(V'(x))^4 \right) f^{(6)}(\frac{\xi^2}2+V(x)) \
  dxd\xi,
  \end{equation*}
This can also be written in a more compact form as
  \begin{equation*}
  \frac 1{1152}\int (7V'V'''+\frac{47}5 (V''(x))^2) f^{(4)}(\frac{\xi^2}2+V(x)) \
  dxd\xi.
  \end{equation*}
It follows that
  \begin{equation*}
  \int_{\frac{\xi^2}2+V(x) \le \lambda} (7V'V'''+\frac{47}5 (V''(x))^2)
  dxd\xi,
  \end{equation*}
is spectrally determined.

A similar but much more lengthy computation yields the coefficient of
$\hb^6$, which is given by
  \begin{equation*}
  \frac 1{2880}\!\int\!\!\left(\!\frac {\xi^8}{490}(V'''\!(x))^2\!-\!
  \frac{\xi^6}{63}(V''\!(x))^3\! -\! \frac{\xi^4}{12}
  (V'(x)V''\!(x))^2\!-\!
  \frac {11}{144}(V'(x))^6\!\right)\!f^{(\!9\!)}(\frac{\xi^2}2+V(x))
  dxd\xi.
  \end{equation*}
In other words, the integral
  \begin{equation*}
  \int_{\frac{\xi^2}2+V(x)\le \lambda} \left(\frac {\xi^8}{490}(V'''(x))^2 -
  \frac{\xi^6}{63}(V''(x))^3 - \frac{\xi^4}{12}
  (V'(x)V''(x))^2 - \frac{11}{144}(V'(x))^6 \right)\ dxd\xi
  \end{equation*}
is spectrally determined.


\end{document}